\documentclass{amsart}

\usepackage{amssymb,latexsym,amsmath,extarrows}
\usepackage{graphicx}
\usepackage{xcolor}
\usepackage{float}

\newtheorem*{acknowledgement}{Acknowledgements}
\newtheorem{theorem}{Theorem}[section]
\newtheorem{lemma}[theorem]{Lemma}

\newtheorem{remark}[theorem]{Remark}

\newtheorem{definition}[theorem]{Definition}

\newcommand{\al}{\alpha}

\newcommand{\ga}{\gamma}

\newcommand{\e}{\varepsilon}

\newcommand{\la}{\lambda}

\newcommand{\si}{\sigma}

\newcommand{\eps}{\epsilon}

\newcommand{\cy}{\mathcal Y}

\newcommand{\ZR}{\mathbb{R}}
\newcommand{\ZZ}{\mathbb{Z}}
\newcommand{\ZC}{\mathbb{C}}
\newcommand{\ZT}{\mathbb{T}}
\newcommand{\ZB}{\mathbb{B}}

\newcommand{\eit}{e^{i t \Delta}}
\newcommand{\eir}{e^{i r \Delta}}
\newcommand{\hichi}{\raisebox{0.7ex}{\(\chi\)}}

\begin{document}

\title[Pointwise convergence and multilinear refined Strichartz]{Pointwise convergence of Schr\"odinger solutions and multilinear refined Strichartz estimates}

\author{Xiumin Du}
\address{
Institute for Advanced Study\\
Princeton, NJ}
\email{xdu@math.ias.edu}
\author{Larry Guth}
\address{Department of Mathematics\\
MIT\\
Cambridge MA}
\email{lguth@math.mit.edu}
\author{Xiaochun Li}
\address{Mathematics Department\\
University of Illinois at Urbana-Champaign\\
Urbana IL}
\email{xcli@math.uiuc.edu}
\author{Ruixiang Zhang}
\address{
Institute for Advanced Study\\
Princeton, NJ}
\email{rzhang@math.ias.edu}

\begin{abstract}
We obtain partial improvement toward the pointwise convergence problem of Schr\"odinger solutions, in the general setting of fractal measure. In particular, we show that, for $n\geq 3$, $\lim_{t \to 0} e^{it\Delta}f(x) = f(x)$ almost everywhere with respect to Lebesgue measure for all $f \in H^s (\mathbb{R}^n)$ provided that $s>(n+1)/2(n+2)$. The proof uses linear refined Strichartz estimates. We also prove a multilinear refined Strichartz using decoupling and multilinear Kakeya.
\end{abstract}

\maketitle

\section{Introduction}
The solution to the free Schr\"{o}dinger equation
\begin{equation}
  \begin{cases}
    iu_t - \Delta u = 0, &(x,t)\in \mathbb{R}^n \times \mathbb{R} \\
    u(x,0)=f(x), & x \in \mathbb{R}^n
  \end{cases}
\end{equation}
is given by
$$
  e^{it\Delta}f(x)=(2\pi)^{-n}\int e^{i\left(x\cdot\xi+t|\xi|^2\right)}\widehat{f}(\xi) \, d\xi.
$$

In \cite{lC}, Carleson proposed the problem of identifying the optimal $s$ for which $\lim_{t \to 0}e^{it\Delta}f(x)=f(x)$ almost everywhere whenever
$f\in H^s(\mathbb{R}^n),$ and proved convergence
for $s \geq 1/4$ when $n=1$. Dahlberg and Kenig \cite{DK} then showed that this result is sharp. The higher dimensional case has since been studied by several authors. In particular, almost everywhere convergence holds if $s>1/2-1/(4n)$ when $n\geq 2$ ($n=2$ due to Lee \cite{sL} and $n\geq 2$ due to Bourgain \cite{jB12}). Recently Bourgain \cite{jB16} gave counterexamples showing that convergence can fail if $s<\frac{n}{2(n+1)}$. Since then, the first three authors \cite{DGL} improved the sufficient condition when $n=2$ to the almost sharp $s>1/3$.

In this article, we obtain the following partial improvement in higher dimensions:
\begin{theorem}\label{thm-PC}
Let $n\geq 3$. For every  $f\in H^s(\mathbb R^n)$ with $s>\frac{n+1}{2(n+2)}$, $\lim_{t \to 0}e^{it\Delta}f(x)=f(x)$ almost everywhere.
\end{theorem}

A natural refinement of Carleson's problem was initiated by Sj\"ogren
and Sj\"olin \cite{SS}: determine the size of divergence set, in particular, consider
$$
\al_n(s):= \sup_{f\in H^s(\ZR^n)} {\rm dim} \left\{x\in \ZR^n: \lim_{t \to 0}e^{it\Delta}f(x)\neq f(x) \right\}\,,
$$
where ${\rm dim}$ stands for the Hausdorff dimension.
Note that when $s>n/2$ the solution is continuous and so $\al_n(s)=0$. Various counterexamples were constructed and in summary the previous results yield
$$
\alpha_n(s) \geq 
\begin{cases}
n, & s<\frac{n}{2(n+1)} \quad \text{(Bourgain \cite{jB16})}\\
n+\frac{n}{n-1}-\frac{2(n+1)s}{n-1}, & \frac{n}{2(n+1)}\leq s<\frac{n+1}{8} \quad \text{(Luc\`a-Rogers \cite{LR17'})}\\
n+1-\frac{2(n+2)s}{n}, & \frac{n+1}{8}\leq s<\frac{n}{4} \quad \text{(Luc\`a-Rogers \cite{LR17})}\\
n-2s, &\frac{n}{4}\leq s\leq \frac{n}{2} \quad \text{(\v{Z}ubrini\'c \cite{Z})}.
\end{cases}
$$
And the previous best known upper bounds are
$$
\al_n(s)\leq
\begin{cases}
n+1-(2+\frac{2}{2n-1})s, &\frac{1}{2}-\frac{1}{4n}<s\leq 1-\frac{3}{2(n+1)} \quad (\text{Luc\`a-Rogers \cite{LR}})
\\
n+1-\frac{1}{n+1}-2s, &1-\frac{3}{2(n+1)}\leq s<\frac{n}{4} \quad (\text{Luc\`a-Rogers \cite{LR}})
\\
n-2s, & \frac{n}{4}\leq s\leq \frac{n}{2} \quad \text{(Barcel\'o-Bennett-Carbery-Rogers \cite{BBCR})}.
\end{cases}
$$
The case $n=1$ has been solved completely. In higher dimensions, we improve Luc\`a-Rogers' result:
\begin{theorem} \label{thm-PC-al}
Let $n\geq 3$. Then
\begin{equation}
\al_n(s)\leq n+1-(2+\frac{2}{n+1})s, \quad
\frac{n+1}{2(n+2)}<s<n/4\,.
\end{equation}
\end{theorem}

\begin{remark}
Theorem \ref{thm-PC-al} also holds when $n=2$ and that recovers the previous results of Lee \cite{sL}, Bourgain \cite{jB12} and Luc\'a-Rogers \cite{LR}, by a different method. In \cite{DGL}, the almost sharp result $s>1/3$ is obtained in the setting of Lebesgue measure, and the sharp Schr\"odinger maximal estimate in \cite{DGL} implies directly the following generalized improvement:
\begin{equation}
\al_2(s)\leq 3-3s,\quad 1/3 <s<1/2.
\end{equation}
\end{remark}

Note that Theorem \ref{thm-PC} is a special case of Theorem \ref{thm-PC-al}. By standard arguments, an upper bound for $\al_n(s)$ can be obtained from appropriate maximal estimates with respect to fractal measure (see for example \cite{LR}). More precisely,
\begin{definition}
Let $\al\in(0,n]$. We say that $\mu$ is $\al$-dimensional if it is a probability measure supported in the unit ball $B^n(0,1)$ and satisfies that
\begin{equation}
\mu(B(x,r))\leq C_\mu r^\al, \quad \forall r>0, \quad \forall x\in \ZR^n.
\end{equation}
\end{definition}

\begin{lemma}[Luc\`a-Rogers, Lemma 7.1 in \cite{LR}] \label{sme-pc}
Let $\al> \al_0\geq n-2s$ and suppose that 
$$
\big\|\sup_{0<t<1} |\eit f|\big\|_{L^1(d\mu)} \leq C_\mu \|f\|_{H^s(\ZR^n)}\,,
$$
whenever $f\in H^s(\ZR^n)$ and $\mu$ is $\al$-dimensional. Then $\al_n(s)\leq \al_0$.
\end{lemma}

In view of Lemma \ref{sme-pc}, it suffices to prove the following Schr\"odinger maximal estimate w.r.t. fractal measure:

\begin{theorem} \label{sme-1}
Let $n\geq 3$ and $s>\frac{\al+1}{2(n+2)}+\frac{n-\al}{2}$. Then
$$
\big\|\sup_{0<t<1} |\eit f|\big\|_{L^2(d\mu)} \leq C_\mu \|f\|_{H^s(\ZR^n)}\,,
$$
whenever $f\in H^s(\ZR^n)$ and $\mu$ is $\al$-dimensional.
\end{theorem}

Denote $d\mu_R(x):=R^\al d\mu(x/R)$. We write $A\lessapprox B$ if $A\leq C_\e R^\e B$ for any $\e>0$.
By a localization argument (see \cite[Lemma 2.3]{sL}), Littlewood-Paley decomposition and parabolic rescaling, Theorem \ref{sme-1} can be reduced to the following:

\begin{theorem} \label{thm-linL2-al}
Let $n\geq 3,\al\in(0,n]$ and $\mu$ be $\al$-dimensional. Then
\begin{equation} \label{linL2-al}
\big\|\sup_{0<t<R} |\eit f|\big\|_{L^2(d\mu_R)} \lessapprox R^{\frac{\al+1}{2(n+2)}} \|f\|_2\,,
\end{equation}
whenever $R\geq 1$ and $f$ has Fourier support in $A(1):=\{\xi\in \ZR^n:|\xi|\sim 1\}$.
\end{theorem}

The key ingredient in our proof is linear refined Strichartz estimate. Linear and bilinear refined Strichartz were derived in \cite{DGL} to solve the pointwise convergence problem in two dimensions. In \cite{DGOWWZ}, via polynomial partitioning developed in \cite{lG,lG16} and linear and bilinear refined Strichartz, some new weighted restriction estimates were established, and as applications improved results were obtained for the Falconer distance set problem and the spherical average decay rates of the Fourier transform of fractal measures. In this article, we prove a multilinear refined Strichartz (see Theorem \ref{thm-kRS}) using decoupling and multilinear Kakeya. The multilinear refined Strichartz may have its own interest. It is also interesting to think about how to exploit this estimate to further improve the weighted restriction and the Schr\"odiner maximal estimates in higher dimensions.

In Section \ref{wpd}, we recall wave packet decomposition briefly. We prove Theorem \ref{thm-linL2-al} in Section \ref{sec:LRS} using linear refined Strichartz estimate. In Section \ref{sec:kRS} we prove a multilinear refined Strichartz.

\begin{acknowledgement}
The work of X.D. is supported by the National Science Foundation under Grant No. 1638352 and the Shiing-Shen Chern Fund. L.G. is supported by a Simons Investigator Award. The work of R.Z. is supported by the National Science Foundation under Grant No. 1638352 and the James D. Wolfensohn Fund.
\end{acknowledgement}

\section{Wave packet decomposition} \label{wpd}
\setcounter{equation}0

We use the same setup as in \cite{lG16,DGL}, which we briefly recall. Let $f$ be a function with Fourier support in the unit ball $B^{n}(0,1)$. We break up $f$ into pieces $f_{\theta,\nu}$ that are localized in both position and frequency.
Cover $B^n(0,1)$ by finitely overlapping balls $\theta$ of radius $R^{-1/2}$ and cover $\ZR^n$ by finitely overlapping balls of radius $R^{\frac{1+\delta}{2}}$, centered at $\nu \in R^{\frac{1+\delta}{2}}\ZZ^n$. Here $\delta=\e^2$ is a small parameter. Using partition of unity, we have a decomposition
$$
f=\sum_{(\theta,\nu)\in \ZT} f_{\theta,\nu} + {\rm RapDec}(R)\|f\|_2\,,
$$
where $f_{\theta,\nu}$ is Fourier supported in $\theta$ and has physical support essentially in a ball of radius $R^{1/2+\delta}$ around $\nu$. The functions $f_{\theta,\nu}$ are approximately orthogonal. For any set $\ZT'$ of pairs $(\theta,\nu)$, we have
$$
\big\|\sum_{(\theta,\nu)\in \ZT'} f_{\theta,\nu}\big\|_2^2
\sim \sum_{(\theta,\nu)\in \ZT'} \|f_{\theta,\nu}\|_2^2\,.
$$
For each pair $(\theta,\nu)$, the restriction of $\eit f_{\theta,\nu}$ to $B^{n+1}_R$ is essentially supported on a tube $T_{\theta,\nu}$ with radius $R^{1/2+\delta}$ and length $R$, with direction $G(\theta)\in S^{n}$ determined by $\theta$ and location determined by $\nu$, more precisely,
$$
T_{\theta,\nu} :=\left\{(x,t) \in B^{n+1}_R : |x+2t\omega_\theta -\nu|\leq R^{1/2+\delta}\right\}\,.
$$
Here $\omega_\theta \in B^n(0,1)$ is the center of $\theta$, and 
$$
G(\theta)=\frac{(-2\omega_\theta,1)}{|(-2\omega_\theta,1)|}\,.
$$

In our discussion of refined Strichartz estimates, we will use the concept of a wave packet being tangent to an algebraic variety. Let $m$ be a dimension in the range $1\leq m\leq n+1$. We write $Z(P_1,\cdots,P_{n+1-m})$ for the set of common zeros of the polynomials $P_1,\cdots,P_{n+1-m}$ on $\ZR^{n+1}$. The variety $Z(P_1,\cdots,P_{n+1-m})$ is a transverse complete intersection if 
$$
\nabla P_1(x) \wedge \cdots \wedge \nabla P_{n+1-m}(x) \neq 0 \text{ for all } x\in Z(P_1,\cdots,P_{n+1-m})\,.
$$
Suppose that $Z$ is an algebraic variety. For any tile $(\theta,\nu) \in\ZT$,
we say that $T_{\theta,\nu}$ is \emph{$E R^{-1/2}$-tangent} to $Z$ if
$$T_{\theta,\nu}\subset N_{E R^{1/2}}Z \cap B^{n+1}_R,\quad and$$
\begin{equation*}
 \text{Angle}(G(\theta),T_zZ)\leq E R^{-1/2}
\end{equation*}
for any non-singular point $z\in N_{2 E R^{1/2}} ( T_{\theta,\nu}) \cap 2B^{n+1}_R \cap Z$.

Let
$$
\ZT_Z (E):=\{(\theta,\nu)\,|\,T_{\theta,\nu} \text{ is $E R^{-1/2}$-tangent to}\, Z\}\,,
$$
and  we say that $f$ is concentrated in wave packets from $\ZT_Z(E)$ if
$$
 \sum_{(\theta,\nu)\notin \ZT_Z(E)} \|f_{\theta,\nu}\|_2 \leq {\rm RapDec}(R)\|f\|_2.
$$
Since the radius of $T_{\theta, \nu}$ is $R^{1/2 + \delta}$, $R^\delta$ is the smallest interesting value of $E$.

\section{Linear refined Strichartz and proof of Theorem \ref{thm-linL2-al}}
\label{sec:LRS}
\setcounter{equation}0

In this section, we prove Theorem \ref{thm-linL2-al} using linear refined Strichartz estimates developed in \cite{DGL}.

\begin{theorem} [Linear refined Strichartz in dimension $n+1$] \label{thm-LRS}

Let $p_{n+1}=\frac{2(n+2)}{n}$. Suppose that $f: \ZR^n  \rightarrow \ZC$ has frequency supported in $B^n(0,1)$.  Suppose that $Q_1, Q_2, ...$ are lattice  $R^{1/2}$-cubes in $B_R^{n+1}$, so that
$$
\| \eit f \|_{L^{p_{n+1}}(Q_j)}  \textrm{ is essentially constant in $j$}. 
$$
Suppose that these cubes are arranged in horizontal slabs of the form $\ZR \times \cdots \times \ZR\times \{t_0, t_0 + R^{1/2} \}$, and that each such slab contains $\sim \sigma$ cubes $Q_j$.  Let $Y$ denote $\bigcup_j Q_j$.  Then
for any $\eps > 0$, 
\begin{equation} \label{LRS}
\| \eit f \|_{L^{p_{n+1}}(Y)} \le C_\e R^\e \sigma^{-\frac{1}{n+2}} \| f \|_{L^2}. \end{equation}
\end{theorem}

\begin{figure}[H]
\centering
\includegraphics[scale=.6]{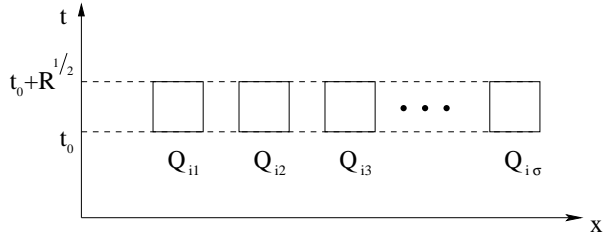}
\caption{\small{ $\sim \sigma$ many cubes in a horizontal slab }}
\label{figure:Fig-strip}
\end{figure}

Theorem \ref{thm-LRS} was proved in \cite{DGL} in dimension $2$, using Bourgain-Demeter $l^2$-decoupling theorem \cite{BD} and induction on scales. The proof of Theorem \ref{thm-LRS} in higher dimensions is similar and we will present the proof in Section \ref{sec:kRS}.

It follows from the Strichartz inequality that $\| \eit f \|_{L^{p_{n+1}}(Y)} \lesssim \| f \|_{L^2}$.  We get an
improvement when $\sigma$ is large.  The condition that $\sigma$ is large forces the solution $\eit f$ to be spread out in space. 

This linear refined Strichartz estimate is sharp. To see this, consider the following example.  Suppose that $\eit f$ is a sum of $\sigma$ wave packets supported on disjoint $R^{1/2} \times \cdots 
\times R^{1/2}\times R$-tubes.  We can take $Y$ to be the union of these tubes.  By scaling, we can suppose that $| \eit f | \sim 1$ on these $\sigma$ tubes and negligibly small elsewhere, and then a direct calculation shows that $\| \eit f \|_{L^{p_{n+1}}(Y)} \sim \| \eit f \|_{L^{p_{n+1}}(B^{n+1}_R)} \sim \sigma^{-1/(n+2)} \| f \|_{L^2}$.
So Theorem \ref{thm-LRS} roughly says that if $\eit f$ is ``as spread out as'' $\sigma$ disjoint wave packets, then its $L^{p_{n+1}}$ norm cannot be much bigger than the $L^{p_{n+1}}$ norm of $\sigma$ disjoint wave packets.

Now we prove Theorem \ref{thm-linL2-al} using linear refined Strichartz estimate:
\begin{proof}[Proof of Theorem \ref{thm-linL2-al}]

Let $n\geq 3, \al\in(0,n]$ and $\mu$ be $\al$-dimensional. We will show that
\begin{equation} \label{linL2-al'}
\big\|\sup_{0<t<R} |\eit f|\big\|_{L^2(d\mu_R)} \lessapprox R^{\frac{\al+1}{2(n+2)}} \|f\|_2
\end{equation}
holds for all $R\geq 1$ and all $f$ with Fourier support in $A(1):=\{\xi\in \ZR^n:|\xi|\sim 1\}$.

Without loss of generality we assume that $\|f\|_2=1$. Let $H$ be a dyadic number and denote
$$
A_H:=\big\{x\in B^n_R : \sup_{0<t< R} |\eit f(x)|\sim H\big\}\,.
$$
Note that we have a trivial bound $H\lesssim 1$ by H\"older's inequality. We also can assume that $R^{-C}<H$ for a large constant $C$, since the contributions from those $A_H$ with $H\leq R^{-C}$ are negligible. Therefore there are only $\sim \log R$ many relevant $H$ and we have
\begin{equation} \label{dya1}
 \big\| \sup_{0<t< R} |\eit f| \big\|_{L^2(d\mu_R)} \lessapprox H \big(\int_{A_H}d\mu_R(x)\big)^{1/2}, \quad \text{ for some dyadic $H$}\,.
\end{equation}
By viewing $|\eit f(x)|$ essentially as constant on unit balls,  we can cover $A_H$ by projection of a set $X$ described as follows: $X$ is a union of unit balls in $B^n_R\times [0,R]$ satisfying that each vertical thin tube of dimensions $1\times \cdots \times 1 \times R$ contains at most one unit ball in $X$, and 
$$
|\eit f(x)| \sim H \, \, \text{on} \,\,X\,.
$$
Next we decompose $B^n_R\times [0,R]$ into $R^{1/2}$-cubes $Q_j$ and consider those $Q_j$'s which intersect $X$. Let $\cy_{\la,\ga,\si}$ denote the collection of those $Q_j$'s such that 
\begin{align*}
\bullet \,&\text{$Q_j$ contains $\sim \la$ unit balls in $X$}\\
\bullet \,&\|\eit f\|_{L^{p_{n+1}}(Q_j)} \sim \ga \\
\bullet \,&\text{the horizontal $R^{1/2}$-slab containing $Q_j$ contains $\sim \si$ $R^{1/2}$-cubes satisfying} \\ 
&\text{the above two conditions}.
\end{align*}
Define $Y _{\la,\ga,\si}:= \bigcup_{Q_j \in \cy_{\la,\ga,\si}} Q_j$. Note that we can assume
$$
1\leq \la \leq R^{n/2}, \quad R^{-C}\leq  \ga \leq R^{C},\quad 1\leq \si \leq R^{n/2}\,,
$$
where $C$ is a large constant. Therefore there are only $\sim (\log R)^3$ many relevant dyadic $(\la,\ga,\si)$ and by \eqref{dya1} we have
\begin{equation} \label{dya2}
\big\| \sup_{0<t\leq R} |\eit f| \big\|_{L^2(d\mu_R)}
\lessapprox H \big(\int_{A_H\cap {\rm Proj}(Y)}d\mu_R(x)\big)^{1/2}\,,
\end{equation}
where $Y=Y_{\la,\ga,\si}$ for some $(\la,\ga,\si)$. Denote $Y:=\bigcup_{j=1}^N Q_j$, then
\begin{equation} \label{N}
N\lesssim R^{1/2}\si\,.
\end{equation} 
Since $|\eit f(x)|$ is essentially constant on unit balls, we have 
\begin{equation} \label{Lp}
H \big(\int_{A_H\cap{\rm Proj}(Y)}d\mu_R(x)\big)^{1/p_{n+1}} \lesssim \|\eit f(x)\|_{L^{p_{n+1}}(Y,dxdt)}\,.
\end{equation}
Now it follows from \eqref{dya2} and \eqref{Lp} that
\begin{equation*}
\big\| \sup_{0<t\leq R} |\eit f| \big\|_{L^2(d\mu_R)} 
\lessapprox  \| \eit f \|_{L^{p_{n+1}}(Y)} \big(\int_{A_H\cap {\rm Proj}(Y)}d\mu_R(x)\big)^{\frac{1}{n+2}} \,,
\end{equation*}
and by Theorem \ref{thm-LRS}, \eqref{N} and the assumption that $\mu$ is $\al$-dimensional, this is further controlled by
\begin{equation*}
\lessapprox \sigma^{-\frac{1}{n+2}}(NR^{\al/2})^{\frac{1}{n+2}} 
\lesssim \sigma^{-\frac{1}{n+2}}  (\si R^{1/2}R^{\al/2})^{\frac{1}{n+2}}
= R^{\frac{\al+1}{2(n+2)}} \,,
\end{equation*}
as desired.
\end{proof}

\section{Multilinear refined Strichartz estimate}\label{sec:kRS}
\setcounter{equation}0

\begin{definition} We say functions $f_i: \ZR^n  \rightarrow \ZC$, $i=1,2,\cdots,k$, have frequencies $k$-transversely supported  in $B^n(0,1)$, if for any points $\xi_i\in {\rm supp} \widehat{f_i} \subset B^n(0,1)$,
$$
|G(\xi_1)\wedge \cdots \wedge G(\xi_k)| \geq c>0\,,
$$
where $c$ is an absolute constant, and $G(\xi):=\frac{(-2\xi,1)}{|(-2\xi,1)|} \in S^n$.
\end{definition}

\begin{theorem} [$k$-linear refined Strichartz in dimension $n+1$] \label{thm-kRS}
Let $p_{n+1}=\frac{2(n+2)}{n}$ and $2\leq k\leq n+1$. Suppose that $f_i: \ZR^n  \rightarrow \ZC$, $i=1,2,\cdots,k$, have frequencies $k$-transversely  supported in $B^n(0,1)$. Suppose that $Q_1, Q_2,\cdots, Q_N$ are lattice  $R^{1/2}$-cubes in $B_R^{n+1}$, so that
$$
\| \eit f_i \|_{L^{p_{n+1}}(Q_j)}  \textrm{ is essentially constant in $j$, for each $i=1,2,\cdots,k$}. 
$$
Let $Y$ denote $\bigcup_{j=1}^N Q_j$.  Then
for any $\eps > 0$, 
\begin{equation} \label{kRS}
\big\| \prod_{i=1}^{k} |\eit f_i|^{\frac{1}{k}} \big\|_{L^{p_{n+1}}(Y)} \le C_\e R^\e N^{-\frac{k-1}{k(n+2)}} \prod_{i=1}^{k} \|f_i\|_2^{1/k}\,. 
\end{equation}
\end{theorem}

Theorem \ref{thm-kRS}
was proved in \cite{DGL} for the case $k=2$ in dimension $2$. We will first present the proof of the linear refined Strichartz in Subsection \ref{sec:pfLRS}. And then in Subsection \ref{sec:pfMLRS} we prove Theorem \ref{thm-kRS}, by combining the proof of the linear case with a geometric estimate derived from Multilinear Kakeya.

\subsection{Proof of Theorem {\ref{thm-LRS}}} \label{sec:pfLRS}
The proof uses the Bourgain-Demeter $l^2$ decoupling theorem, together with induction on the radius and parabolic rescaling.  
First we recall the decoupling result of Bourgain and Demeter in \cite{BD}.
\begin{theorem}[Bourgain-Demeter]  \label{bourdem} Let $m\geq 2$ and $p_m:=\frac{2(m+1)}{m-1}$.
Suppose that the $R^{-1}$-neighborhood of the unit paraboloid in
$\ZR^{m}$ is divided into $R^{(m-1)/2}$ disjoint rectangular boxes $\tau$, each with dimensions $R^{-1/2}\times \cdots\times R^{-1/2} \times R^{-1}$.
Suppose $\widehat F_\tau$ is supported in $\tau$ and $F = \sum_\tau F_\tau$.
Then
$$ \| F \|_{L^{p_m}(\ZR^m)} \lessapprox \big( \sum_\tau \| F_\tau \|_{L^{p_m}(\ZR^m)}^2 \big)^{1/2}. $$
\end{theorem}

To set up the argument, we decompose $f$ as follows.  We break the unit ball $B^n(0,1)$ in frequency space into
small balls $\tau$ of radius $R^{-1/4}$, and divide the physical space ball $B^n_R$ into balls $B$ of
radius $R^{3/4}$.  For each pair $(\tau, B)$, we let $f_{\Box_{\tau,B}}$ be the function formed by cutting
 off $f$ on the ball $B$ (with a Schwartz tail) in physical space and the ball $\tau$ in Fourier space.
We note that $e^{it\Delta} f_{\Box_{\tau,B}}$, restricted to $B^{n+1}_R$, is essentially supported on an
 $R^{3/4} \times\cdots\times R^{3/4} \times R$-box, which we denote by $\Box_{\tau,B}$. The box $\Box_{\tau,B}$ is in the direction given by $(-2c(\tau),1)$ and intersects ${t=0}$ at a disk centered at $(c(B),0)$, where $c(\tau)$
and $c(B)$ are the centers of $\tau$ and $B$ respectively. For a fixed $\tau$, the different boxes $\Box_{\tau, B}$
tile $B^{n+1}_R$.  In particular, for each $\tau$, a given cube $Q_j$ lies in exactly one box $\Box_{\tau, B}$. Therefore, the decoupling theorem tells us that
\begin{equation}
\| e^{i t \Delta} f \|_{L^{p_{n+1}}(Q)} \lessapprox \big( \sum_\Box \| e^{ i t \Delta} f_\Box \|_{L^{p_{n+1}}(Q)}^2 \big)^{1/2}.
\end{equation}

The second ingredient is induction on the radius. Using parabolic rescaling and induction on the radius, we get a version of our main inequality for each function $f_\Box$. It goes as follows:

Suppose that $S_1, S_2, ...$ are $R^{1/2} \times\cdots\times R^{1/2} \times R^{3/4}$-tubes in $\Box$ (running parallel to the long axis of $\Box$), and that
$$ \| e^{i t \Delta} f_\Box \|_{L^{p_{n+1}}(S_j)} \textrm{ is essentially constant in $j$}. $$
Suppose that these tubes are arranged into $R^{3/4}$-slabs running parallel to the short axes of
$\Box$ and that each such slab contains $\sim \sigma_\Box$ tubes $S_j$.  Let $Y_\Box$ denote $\cup_j S_j$.  Then
\begin{equation} \label{indpar} \| e^{i t \Delta} f_\Box \|_{L^{p_{n+1}}(Y_\Box)} \lesssim  R^{\e/2} \sigma_\Box^{-\frac{1}{n+2}} \| f_\Box \|_{L^2}. \end{equation}

To apply this inequality, we need to identify a good choice of $Y_\Box$.  We do this by some dyadic pigeonholing.  For each $\Box$, we apply the following algorithm to regroup tubes in $\Box$:
\begin{enumerate}
\item We sort those $R^{1/2} \times\cdots\times R^{1/2} \times R^{3/4}$-tubes $S$ contained in the box $\Box$
according to the order of magnitude of $\| \eit f_\Box \|_{L^{p_{n+1}}(S)}$, which we denote $\lambda$.
For each dyadic number $\lambda$, we use $\mathbb{S}_\lambda$ to stand for the collection of tubes
$S \subset \Box$ with $\| \eit f_\Box \|_{L^{p_{n+1}}(S)} \sim \lambda$.
\item For each $\lambda$, we sort the tubes $S \in \mathbb{S}_\lambda$ by looking at the number of such
 tubes in an $R^{3/4}$-slab.  For any dyadic number $\eta$, we let $\mathbb{S}_{\lambda, \eta}$ be the set of tubes $S \in \mathbb{S}_\lambda$ so that the number of tubes of $\mathbb{S}_\lambda$ in the $R^{3/4}$-slab containing $S$ is $\sim \eta$.
\end{enumerate}

\begin{figure}  [H]
\centering
\includegraphics[scale=.3]{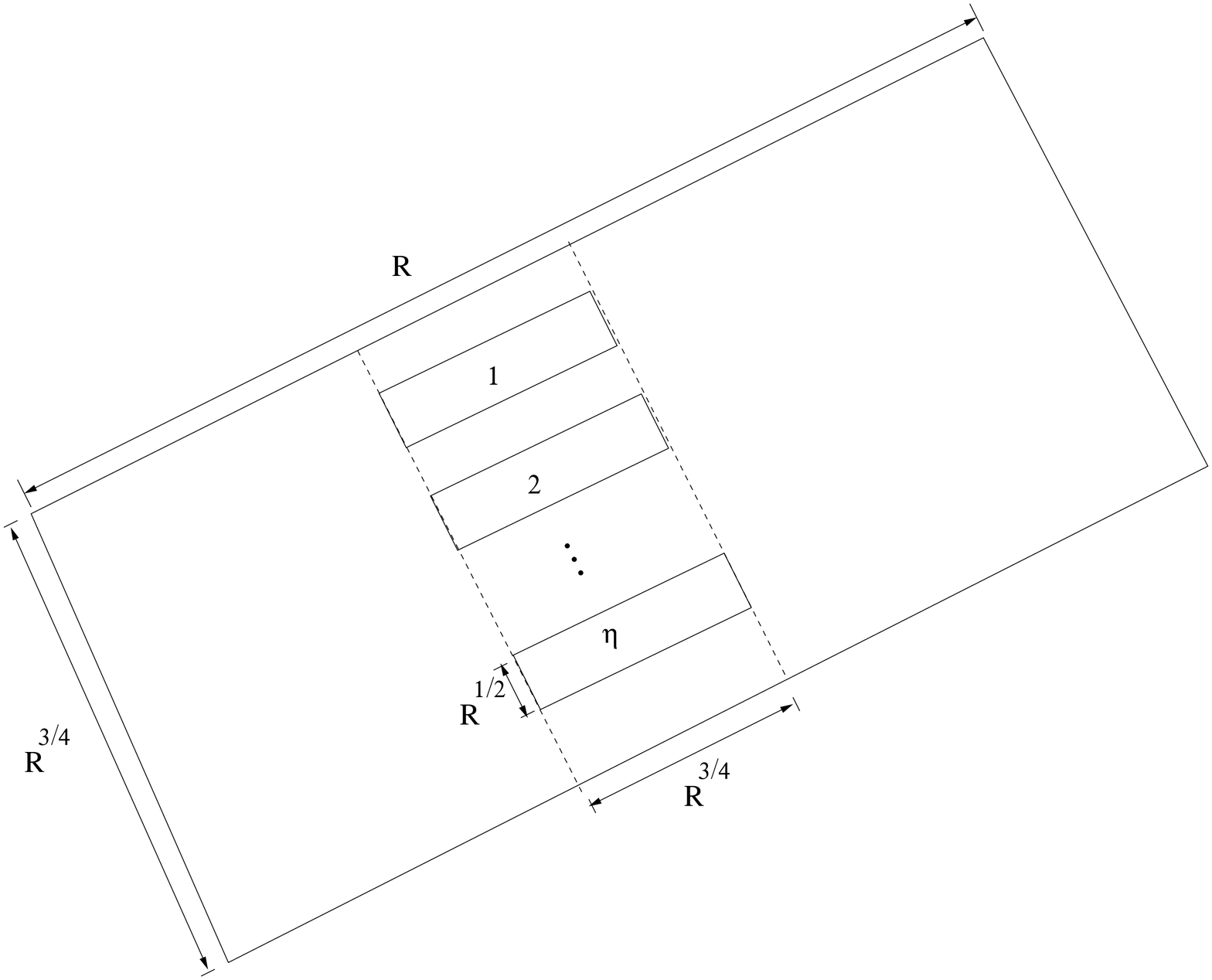}
\caption{\small{ Tubes in a given slab in the $\Box$  }}
\label{figure:Fig-box}
\end{figure}

We let $Y_{\Box, \lambda, \eta}$ be the union of the tubes in $\mathbb{S}_{\lambda, \eta}$. Then we represent
$$ \eit f = \sum_{\lambda, \eta} \big( \sum_\Box \eit f_\Box \cdot \hichi_{Y_{\Box, \lambda, \eta}} \big). $$
Since there are $\lesssim \log R$ choices for each of $\lambda, \eta$, we can choose $\lambda, \eta$ so that
\begin{equation} \label{dyadpig1}
\| \eit f \|_{L^{p_{n+1}}(Q_j)} \lesssim (\log R)^2 \big \| \sum_\Box \eit f_\Box \cdot\hichi_{Y_{\Box, \lambda, \eta}}\big \|_{L^{p_{n+1}}(Q_j)}
 \end{equation}
holds for a fraction $\approx 1$ of all cubes $Q_j$ in $Y$.
We need this uniform choice of $(\lambda, \eta)$, which is independent of $Q_j$,  because later we will sum over all $Q_j$ and arrive at $\|\eit f_\Box\|_{L^{p_{n+1}}(Y_{\Box,\lambda, \eta})}$. 

We fix $\lambda$ and $\eta$ for the rest of the proof.
Let $Y_\Box$ stand for the abbreviation of $Y_{\Box, \lambda, \eta}$.
We note that $Y_\Box$ obeys the hypotheses for our inductive estimate (\ref{indpar}), with $\sigma_\Box$ being the value of $\eta$ that we have fixed.

The following geometric estimate will play a crucial role in our proof.  Each set $Y_\Box$ contains $\lesssim \sigma_\Box$ tubes in each slab parallel to the short axes of $\Box$.  Since the angle between the short axes of $\Box$ and the $x$-axes is bounded away from $\pi/2$, it follows that $Y_\Box$ contains $\lesssim \sigma_\Box$ cubes $Q_j$ in any $R^{1/2}$-horizontal row. Therefore,
\begin{equation} \label{geom1}  | Y_\Box \cap Y | \lesssim \frac{\sigma_\Box}{\sigma} |Y|. \end{equation}

Next we sort the the boxes $\Box$ according to the dyadic size of $\| f_\Box \|_{L^2}$.  We can restrict matters
to $\lesssim \log R$ choices of this dyadic size, and so we can choose a set of $\Box$'s, $\mathbb{B}$,
so that $\| f_\Box \|_{L^2}$ is essentially constant for $\Box \in \mathbb{B}$ and
\begin{equation} \label{dyadpig2}
\| \eit f \|_{L^{p_{n+1}}(Q_j)} \lessapprox  \| \sum_{\Box \in \mathbb{B}} \eit f_\Box\cdot \hichi_{Y_{\Box}} \|_{L^{p_{n+1}}(Q_j)}
\end{equation}
for a fraction $\approx 1$ of cubes $Q_j$ in $Y$.

Finally we sort the cubes $Q_j \subset Y$ according to the number of $Y_\Box$ that contain them.
 We let $Y' \subset Y$ be a set of cubes $Q_j$ which obey (\ref{dyadpig2}) and which each lie in $\sim \mu$ of
the sets $\{ Y_\Box \}_{\Box \in \mathbb{B}}$.  Because (\ref{dyadpig2}) holds for a large fraction of cubes,
and because there are only dyadically many choices of $\mu$, $|Y'| \approx |Y|$.  By the equation (\ref{geom1}), we
see that
$$  | Y_\Box \cap Y' | \le |Y_\Box \cap Y| \lessapprox  \frac{\sigma_\Box}{\sigma} |Y| \approx \frac{\sigma_\Box}{\sigma} |Y'|. $$
Therefore, the multiplicity $\mu$ is bounded by
\begin{equation}\label{geom2}
\mu \lessapprox \frac{\sigma_\Box}{\sigma} | \mathbb{B} |.
\end{equation}

We now are ready to combine all our ingredients and finish our proof.  By decoupling, we have for each $Q_j \subset Y'$,
\begin{equation}
\begin{split}
\| \eit f \|_{L^{p_{n+1}}(Q_j)} \lessapprox &\big\| \sum_{\Box \in \mathbb{B}} \eit f_\Box \cdot\hichi_{Y_\Box} \big\|_{L^{p_{n+1}}(Q_j)}\\
\lessapprox & \big(   \sum_{\Box \in \mathbb{B}\,: \,Q_j \subset Y_\Box} \big\| \eit f_\Box \big\|_{L^{p_{n+1}}(Q_j)}^2 \big)^{1/2}.
\end{split}
\end{equation}
Since the number of $Y_\Box$ containing $Q_j$ is $\sim \mu$, we can apply H\"older to get 
$$  \big\| \sum_{\Box \in \mathbb{B}} \eit f_\Box \cdot \hichi_{Y_\Box} \big\|_{L^{p_{n+1}}(Q_j)} \lessapprox \mu^{\frac{1}{n+2}}
\big(   \sum_{\Box \in \mathbb{B}\,:\, Q_j \subset Y_\Box} \big\| \eit f_\Box \big\|_{L^{p_{n+1}}(Q_j)}^{p_{n+1}} \big)^{1/p_{n+1}}. $$
Now we raise to the $p_{n+1}$-th power and sum over $Q_j \subset Y'$ to get
$$  \left\| \eit f \right\|_{L^{p_{n+1}}(Y')}^{p_{n+1}} \lessapprox \mu^{\frac{2}{n}} \sum_{\Box \in \mathbb{B}}  \big\| \eit f_\Box
\big\|_{L^{p_{n+1}}(Y_\Box)}^{p_{n+1}}. $$
Since $|Y'| \gtrapprox |Y|$, and since each cube $Q_j \subset Y$ makes an equal contribution to $\| \eit f \|_{L^{p_{n+1}}(Y)}$, we see that
$\| \eit f \|_{L^{p_{n+1}}(Y)} \approx \| \eit f \|_{L^{p_{n+1}}(Y')}$ and so
$$  \left\| \eit f \right\|_{L^{p_{n+1}}(Y)}^{p_{n+1}} \lessapprox  \mu^{\frac{2}{n}} \sum_{\Box \in \mathbb{B}}
\left\| \eit f_\Box \right\|_{L^{p_{n+1}}(Y_\Box)}^{p_{n+1}}. $$

By a parabolic rescaling, Figure \ref{figure:Fig-box} becomes Figure \ref{figure:Fig-cubes}.
\begin{figure}  [H]
\centering
\includegraphics[scale=.4]{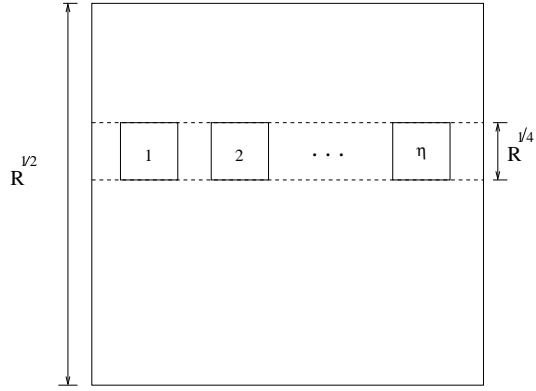}
\caption{\small{ Cubes in a given slab in an $R^{1/2}$-cube }}
\label{figure:Fig-cubes}
\end{figure}

\noindent Henceforth, applying our inductive hypothesis (\ref{indpar}) at scale $R^{1/2}$ to the right-hand side, we see that
\begin{equation} \label{forthm2}
 \left\| \eit f \right\|_{L^{p_{n+1}}(Y)}^{p_{n+1}}  \lessapprox \mu^{\frac{2}{n}} \sigma_\Box^{-\frac{2}{n}} \sum_{\Box \in \mathbb{B}}
\| f_\Box \|_{L^2}^{p_{n+1}}.
\end{equation}
Plugging in our bound for $\mu$ in (\ref{geom2}), this is bounded by
$$  \lesssim \sigma^{-\frac{2}{n}} | \mathbb{B}|^{\frac{2}{n}} \sum_{\Box \in \mathbb{B}} \| f_\Box \|_{L^2}^{p_{n+1}}. $$
Now since $\| f_\Box \|_{L^2}$ is essentially constant among all $\Box \in \mathbb{B}$, the last expression is
 $$ \sim \sigma^{-\frac{2}{n}} (\sum_{\Box \in \mathbb{B}} \| f_\Box \|_{L^2}^2)^{p_{n+1}/2} \le \sigma^{-\frac{2}{n}}
 \| f \|_{L^2}^{p_{n+1}}. $$
Taking the $p_{n+1}$-th root, we obtain our desired bound: 
$$ \| \eit f \|_{L^{p_{n+1}}(Y)} \lessapprox \sigma^{-\frac{1}{n+2}} \| f \|_{L^2}. $$
This closes the induction on radius and completes the proof.
 
\subsection{Proof of Theorem \ref{thm-kRS}} \label{sec:pfMLRS}

One key ingredient in our proof is Bennett-Carbery-Tao multilinear Kakeya estimates:

\begin{theorem} [see \cite{BCT} and \cite{G}] \label{MK}
Suppose that $S_j \subset S^{m-1}, j=1,\cdots,k$. Suppose that $l_{j,a}$ are lines in $\ZR^m$ and that the direction of $l_{j,a}$ lies in $S_j$. Suppose that for any vectors $v_j\in S_j$,
$$
|v_1\wedge \cdots \wedge v_k|\geq \nu.
$$
Let $T_{j,a}$ be the characteristic function of the $1$-neighborhood of $l_{j,a}$. Let $Q_s$ denote any cube of side length $S$. Then for any $\e>0$ and any $S\geq 1$, there holds
$$
\int_{Q_s} \prod_{j=1}^{k} \big(\sum_{a=1}^{N_j} T_{j,a}\big)^{1/(k-1)} \leq C_\e {\rm Poly}(\nu^{-1}) S^{\e} \prod_{j=1}^{k} N_j^{1/(k-1)}\,.
$$
\end{theorem}

Now we begin the proof of Theorem \ref{thm-kRS}. By H\"older,
$$ 
\big\| \prod_{i=1}^{k} |\eit f_i|^{1/k} \big\|_{L^{p_{n+1}}(Y)} \le \prod_{i=1}^k \big\| e^{i t \triangle} f_i  \big\|_{L^{p_{n+1}}(Y)}^{1/k}. 
$$
For each $i$, we process $\| \eit f_i \|_{L^{p_{n+1}}(Y)}$ following the proof of Theorem \ref{thm-LRS}. We decompose $f_i = \sum_\Box f_{i, \Box}$, and we follow the proof of Theorem \ref{thm-LRS} up to equation (\ref{forthm2}).  Therefore, for each $i$, we see that
\begin{equation} \label{forthm2'}
 \big\| \eit f_i \big\|_{L^{p_{n+1}}(Y)} \lessapprox \left[\mu_i^{\frac{2}{n}} \sigma_{i,\Box}^{-\frac{2}{n}} \sum_{\Box \in \mathbb{B}_i}
\| f_{i,\Box} \|_{L^2}^{p_{n+1}} \right]^{1/p_{n+1}}.
\end{equation}
We claim that the following geometric estimate holds:
\begin{equation} \label{geom-k} 
N \prod_{i=1}^k \mu_i^{1/(k-1)} \lessapprox \prod_{i=1}^k \big(\sigma_{i, \Box}  | \mathbb{B}_i | \big)^{1/(k-1)}. 
\end{equation}
Starting with (\ref{forthm2'}) and inserting this estimate, we see that 
\begin{equation*}
\begin{split}
&\prod_{i=1}^k \big\| e^{i t \triangle} f_i  \big\|_{L^{p_{n+1}}(Y)}^{1/k}
\lessapprox \prod_{i=1}^k \left[ \mu_i^{\frac 2n} \sigma_{i, \Box}^{-\frac 2n} \sum_{\Box \in \mathbb{B}_i} \| f_{i, \Box} \|_{L^2}^{p_{n+1}} \right]^{\frac{1}{p_{n+1}} \cdot \frac{1}{k} } \\
\lessapprox &\prod_{i=1}^k \left[ N^{-\frac{2(k-1)}{kn}} | \mathbb{B}_i|^{\frac 2 n}  \sum_{\Box \in \mathbb{B}_i} \| f_{i, \Box} \|_{L^2}^{p_{n+1}} \right]^{\frac{1}{p_{n+1}} \cdot \frac{1}{k} } \lesssim N^{-\frac{k-1}{k(n+2)}} \prod_{i=1}^k \| f_i \|_{L^2}^{1/k}\,.
\end{split}
\end{equation*}
where the last inequality follows from the assumption that $\|f_{i,\Box}\|_{L^2}$ is essentially constant among all $\Box\in \ZB_i$.
It remains to prove the claim \eqref{geom-k}. See Figure \ref{figure:Fig-trans} to get some intuition about how two transversal families of tubes intersect. For the higher order of linearity, we need to invoke multilinear Kakeya estimates - Theorem \ref{MK}. 

\begin{figure}  [H]
\centering
\includegraphics[scale=.5]{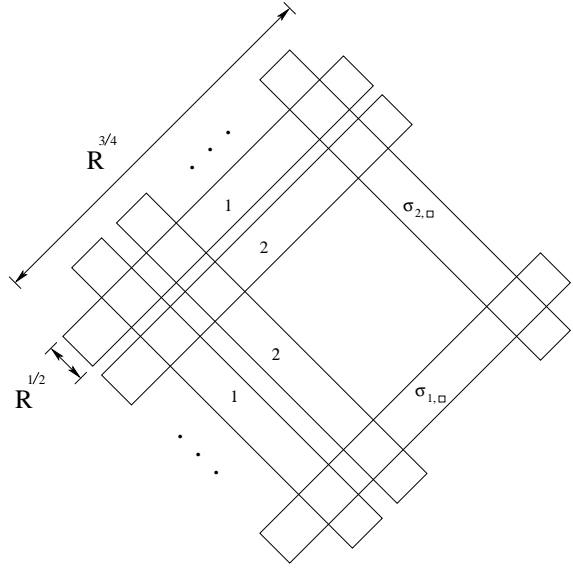}
\caption{\small{ at most $O(\sigma_{1, \Box}\sigma_{2, \Box})$ cubes created by two transversal families of
  rectangular boxes  }}
\label{figure:Fig-trans}
\end{figure}

Recall that $Y'\subset Y, |Y|\lessapprox |Y'|$, the number of $R^{1/2}$-cubes in $Y$ is $N$, and for each $Q$ in $Y'$,
$$
\#\{\Box \in \ZB_i: Q \subset Y_\Box\} \sim \mu_i\,.
$$
Therefore,
$$
N\prod_{i=1}^k \mu_i^{1/(k-1)} \lessapprox
\sum_{Q\in Y'}\prod_{i=1}^k \big(\#\{\Box \in \ZB_i: Q \subset Y_\Box\}\big)^{1/(k-1)}\,.
$$
Cover $B^{n+1}_R$ by balls $B$ of radius $R^{3/4}$. Observe that if an $R^{1/2}$-cube $Q$ inside $B$ is contained in some $Y_\Box$, then $B$ is contained in $10\Box$. Define
$$
\ZB_{i,B}:=\{\Box\in \ZB_i : B\in 10 \Box\}\,,
$$
then
$$
N\prod_{i=1}^k \mu_i^{1/(k-1)} \lessapprox
\sum_{B:R^{3/4}\text{-balls}} \,\,\sum_{Q\in Y':Q\subset B} \,\,\prod_{i=1}^k \big( \#\{\Box\in \ZB_{i,B}: Q \in Y_\Box\}
\big)^{1/(k-1)}\,.
$$
Note that for each $B$, we have $k$ transverse collections of $R^{1/2}\times\cdots\times R^{1/2}\times R^{3/4}$-tubes passing through it, and the number of such tubes in the $i$-th collection is $\lesssim |\ZB_{i,B}|\cdot \sigma_{i,\Box}$. It follows from the multilinear Kakeya estimate that
$$
\sum_{Q\in Y':Q\subset B} \,\,\prod_{i=1}^k \big( \#\{\Box\in \ZB_{i,B}: Q \in Y_\Box\}
\big)^{1/(k-1)} \lessapprox \prod_{i=1}^k \big(|\ZB_{i,B}|\cdot \sigma_{i,\Box}\big)^{1/(k-1)}\,.
$$
Therefore,
$$
N\prod_{i=1}^k \mu_i^{1/(k-1)} \lessapprox
\sum_{B:R^{3/4}\text{-balls}} \,\,\prod_{i=1}^k \big(|\ZB_{i,B}|\cdot \sigma_{i,\Box}\big)^{1/(k-1)}\,.
$$
By the definition of $\ZB_{i,B}$ and multilinear Kakeya again, 
$$
\sum_{B:R^{3/4}\text{-balls}} \,\,\prod_{i=1}^k |\ZB_{i,B}|^{1/(k-1)} \lessapprox \prod_{i=1}^k |\ZB_i|^{1/(k-1)}\,.
$$
Combining these together, we get the desired estimate \eqref{geom-k}.

\subsection{Refined Strichartz estimates in variety case}
We remark that, by the same technique as in \cite{DGL}, Theorem \ref{thm-LRS} and \ref{thm-kRS} can be generalized to variety case as follows. We skip the rigorous proof and refer interested readers to Section 7 of \cite{DGL}.

\begin{theorem} [Linear refined Strichartz for $m$-variety in dimension $n+1$]  \label{thm-LRS-m}
Let \\
$m$ be a dimension in the range $2\leq m \leq n+1$. Let $p_m=2(m+1)/(m-1)$. Suppose that $Z=Z(P_1,\cdots,P_{n+1-m})$ is a transverse complete intersection where ${\rm Deg}\,P_i \leq D_Z = R^{\delta_{\rm deg}}$. Here $\delta_{\rm deg} \ll \delta$ is a small parameter. Suppose that $f\in L^2(\ZR^{n})$ is Fourier supported in $B^n(0,1)$ and concentrated in wave packets from $\ZT_Z (E).$ 
Suppose that $Q_1, Q_2, ...$
are lattice  $R^{1/2}$-cubes in $B_R$, so that
$$ \| \eit f \|_{L^{p_m}(Q_j)} \textrm{ is essentially constant in $j$}. $$
\noindent Suppose that these cubes are arranged in horizontal slabs of the form $\ZR \times \cdots \times \ZR \times \{t_0, t_0 + R^{1/2} \}$, and that each such slab contains $\sim \sigma$ cubes $Q_j$.  Let $Y$ denote $\bigcup_j Q_j$. Then
\begin{equation}
\label{LRS-m}\| \eit f \|_{L^{p_m}(Y)} \lessapprox  E^{O(1)}R^{-\frac{n+1-m}{2(m+1)}} \sigma^{-\frac{1}{m+1}}  \| f \|_{L^2}.
\end{equation}
\end{theorem}

\begin{theorem} [$k$-linear refined Strichartz for $m$-variety in dimension $n+1$] \label{thm-kRS-m} 
Let \\
$m$ be a dimension in the range $2 \leq m \leq n+1$. Let $p_m=2(m+1)/(m-1)$. Suppose that $f_i: \ZR^n  \rightarrow \ZC$, $i=1,2,\cdots,k$, have frequencies $k$-transversely supported in $B^n(0,1)$, where $2\leq k\leq m$. Suppose that the functions $f_i$ are concentrated in wave packets from $\ZT_Z(E)$, where $Z=Z(P_1,\cdots,P_{n+1-m})$  is a transverse complete intersection with ${\rm Deg}\,P_i \leq D_Z=R^{\delta_{\rm deg}}$. Here $\delta_{\rm deg} \ll \delta$ is a small parameter. Suppose that $Q_1, Q_2,\cdots, Q_N$ are lattice  $R^{1/2}$-cubes in $B_R^{n+1}$, so that
$$
\| \eit f_i \|_{L^{p_m}(Q_j)}  \textrm{ is essentially constant in $j$, for each $i=1,2,\cdots,k$}. 
$$
Let $Y$ denote $\bigcup_{j=1}^{N} Q_j$. Then
\begin{equation} \label{kRS-m}
\left\| \prod_{i=1}^{k} |\eit f_i|^{1/k} \right\|_{L^{p_m}(Y)} \lessapprox E^{O(1)} R^{-\frac{n+1-m}{2(m+1)}} N^{-\frac{k-1}{k(m+1)}} \prod_{i=1}^k \| f_i \|_{L^2}^{1/k}\,.
\end{equation}
\end{theorem}

To get some intuition, we consider a special case of Theorem \ref{thm-LRS-m}, in which the variety $Z$ is
naturally replaced by an $m$-plane $V$, and $E \approx 1$.  In the planar case, all wave packets are contained in the
$\approx R^{1/2}$-neighborhood of $V$, and the absolute value $|\eit f(x)|$ is essentially constant along $(n+1-m)$-planes which are parallel to $V'$, where $V'$ is a subspace transverse (roughly normal) to $V$.
 Note that $\eit f(x)|_V$ is a Fourier extension operator in
dimension $m$. Denote $\eit f(x)|_V$ by $\eir h(y)$ for some function $h$ Fourier supported in $B^{m-1}(0,1)$, where $(y,r)$ denote coordinate variables for $V$.
Hence the conclusion in Theorem \ref{thm-LRS-m} can be rephrased in terms of $h$. Indeed, observe that
 $$\| \eit f (x)\|^{p_m}_{L^{p_m}(Y)} \sim R^{(n+1-m)/2} \| \eir h(y)\|^{p_m}_{L^{p_m}(Y\cap V)},$$
 and
\begin{equation*} 
\begin{split}
\|f\|^2_2\sim &R^{-1}\|\eit f\|_{L^2(B^{n+1}_R)}^2 \\
\sim & R^{-1}R^{(n+1-m)/2}\|\eir h\|^2_{L^2(B^{n+1}_R\cap V)}\sim R^{(n+1-m)/2}
\|h\|_2^2.
\end{split}
\end{equation*}
Therefore the estimate \eqref{LRS-m} is equivalent to
\begin{equation} \label{linrefV}
\|\eir h\|_{L^{p_m}(Y\cap V)} \lessapprox \sigma^{-1/(m+1)}\|h\|_{L^2}.
\end{equation}
This is exactly the conclusion of Theorem \ref{thm-LRS} in dimension $m$. Similarly the $m$-plane case of Theorem \ref{thm-kRS-m} is essentially Theorem \ref{thm-kRS} in dimension $m$.


\begin{thebibliography}{9}
\bibitem{BBCR}
J. A. Barcel\'o, J. Bennett, A. Carbery and K. M. Rogers,
\emph{On the dimension of divergence sets
of dispersive equations}, 
Math. Ann. \textbf{349} (2011), 599-622.

\bibitem{BCT} J. Bennett, A. Carbery and T. Tao, \emph{On the multilinear restriction
and Kakeya conjectures}, 
Acta Math., \textbf{196} (2006), 261-302

\bibitem{jB12}
J. Bourgain,
\emph{On the Schr\"{o}dinger maximal function in higher dimension},
Proceedings of the Steklov Institute of Math. 2013, vol. 280, pp. 46-60. (2012).

\bibitem{jB16}
J. Bourgain,
\emph{A note on the Schr\"odinger maximal function},
J. Anal. Math. 130 (2016), 393-396. 

\bibitem{BD} J. Bourgain and C. Demeter, \emph{The proof of the $l^2$ decoupling conjecture},
Ann. of Math. (2) \textbf{182} (2015), no. 1, 351-389.

\bibitem{lC}
L. Carleson,
\emph{Some analytic problems related to statistical mechanics},
Euclidean Harmonic Analysis (Proc. Sem., Univ. Maryland, College Park, Md, 1979),
Lecture Notes in Math.779, pp. 5-45.

\bibitem{DK}
B.E.J. Dahlberg and C.E. Kenig,
\emph{A note on the almost everywhere behavior of solutions to the Schr\"{o}dinger equation},
Harmonic Analysis (Minneapolis, Minn, 1981), Lecture Notes in Math. 908, pp.205-209.

\bibitem{DGL} X. Du, L. Guth and X. Li, \emph{A sharp Schr\"odinger maximal estimate in $\ZR^2$},
Annals of Math \textbf{186} (2017), 607-640.

\bibitem{DGOWWZ} X. Du, L. Guth, Y. Ou, H. Wang, B. Wilson and R. Zhang,
\emph{Weighted restriction estimates and application to Falconer distance set problem},
preprint (2018), arXiv:1802.10186

\bibitem{G} L. Guth,
\emph{A short proof of the multilinear Kakeya inequality},
Math. Proc. Cambridge Philos. Soc. \textbf{158} (2015), no. 1, 147-153. 

\bibitem{lG}
L. Guth,
\emph{A restriction estimate using polynomial partitioning},
J. Amer. Math. Soc. \textbf{29} (2016), no. 2, 371-413.

\bibitem{lG16}
L. Guth,
\emph{Restriction estimates using polynomial partitioning II},
preprint (2016), arXiv:1603.04250.

\bibitem{sL}
S. Lee,
\emph{On pointwise convergence of the solutions to Schr\"{o}dinger equations in $\mathbb{R}^2$},
International Math. Research Notices. 2006, 32597, 1-21 (2006).

\bibitem{LR} R. Luc\`a and K. Rogers,
\emph{Average decay for the Fourier transform of measures with applications},
J. Eur. Math. Soc. (2016, to appear)

\bibitem{LR17} R. Luc\`a and K. Rogers,
\emph{Coherence on fractals versus convergence for the Schr\"odinger
equation}, Comm. Math. Phys. \textbf{351} (2017), 341-359.

\bibitem{LR17'} R. Luc\`a and K. Rogers,
\emph{A note on pointwise convergence for the Schr\"odinger equation},
preprint (2017), arXiv:1703.01360.

\bibitem{SS} P. Sj\"ogren and P. Sj\"olin, \emph{Convergence properties for the time-dependent Schr¨odinger equation},
Ann. Acad. Sci. Fenn. \textbf{14} (1989), no. 1, 13-25.

\bibitem{Z} D. \v{Z}ubrini\'c,
\emph{Singular sets of Sobolev functions}, C. R. Math. Acad. Sci. Paris \textbf{334}, 539-544 (2002).

\end{thebibliography}
\end{document}